\documentclass[a4paper,reqno,12pt]{amsart}

\usepackage{graphicx}
\usepackage{psfrag}
\usepackage{perpage}
\usepackage{url}

\usepackage{color}
\usepackage{mathrsfs}
\usepackage{dsfont} 
\usepackage{mathdots}
\usepackage{tikz}
\usepackage{young}

\usepackage{amsmath, amsfonts, amssymb, amsthm, amscd}

\usepackage[utf8]{inputenc}
\usepackage[T1]{fontenc}
\usepackage{microtype}

\usepackage[a4paper,scale={0.72,0.74},marginratio={1:1},footskip=7mm,headsep=10mm]{geometry}

\usepackage{hyperref}

\makeatletter
\def\@secnumfont{\bfseries\scshape}

\def\section{\@startsection{section}{1}%
  \z@{.7\linespacing\@plus\linespacing}{.5\linespacing}%
  {\normalfont\large\bfseries\scshape\centering}}

\def\subsection{\@startsection{subsection}{2}%
  \z@{.5\linespacing\@plus.7\linespacing}{-.5em}%
  {\normalfont\bfseries\scshape}}
  
\def\subsubsection{\@startsection{subsubsection}{3}%
  \z@{.5\linespacing\@plus.7\linespacing}{-.5em}%
  {\normalfont\scshape}}

\def\specialsection{\@startsection{section}{1}%
  \z@{\linespacing\@plus\linespacing}{.5\linespacing}%
  {\normalfont\centering\large\bfseries\scshape}}
\makeatother

\makeatletter

\renewenvironment{proof}[1][\proofname]{\par
\pushQED{\qed}%
\normalfont \topsep4\p@\@plus4\p@\relax
\trivlist
\item[\hskip\labelsep
\bfseries
#1\@addpunct{.}]\ignorespaces
}{%
\popQED\endtrivlist\@endpefalse
}
\makeatother

\makeatletter
\newcommand \Dotfill {\leavevmode \leaders \hb@xt@ 6pt{\hss .\hss }\hfill \kern \z@}
\makeatother

\makeatletter
\def\@tocline#1#2#3#4#5#6#7{\relax
  \ifnum #1>\c@tocdepth 
  \else
    \par \addpenalty\@secpenalty\addvspace{#2}%
    \begingroup \hyphenpenalty\@M
    \@ifempty{#4}{%
      \@tempdima\csname r@tocindent\number#1\endcsname\relax
    }{%
      \@tempdima#4\relax
    }%
    \parindent\z@ \leftskip#3\relax \advance\leftskip\@tempdima\relax
    \rightskip\@pnumwidth plus4em \parfillskip-\@pnumwidth
    #5\leavevmode\hskip-\@tempdima
      \ifcase #1
       \or\or \hskip 1.65em \or \hskip 3.3em \else \hskip 4.95em \fi%
      #6\nobreak\relax
    \Dotfill
    \hbox to\@pnumwidth{\@tocpagenum{#7}}\par
    \nobreak
    \endgroup
  \fi}
\makeatother

\makeatletter
\def\l@section{\@tocline{1}{0pt}{1pc}{}{\scshape}}
\renewcommand{\tocsection}[3]{%
\indentlabel{\@ifnotempty{#2}{\ignorespaces#1 #2.\hskip 0.7em}}#3}
\def\l@subsection{\@tocline{2}{0pt}{1pc}{5pc}{}}

\def\l@subsubsection{\@tocline{3}{0pt}{1pc}{7pc}{}}

\makeatother

\numberwithin{equation}{section}

\newtheoremstyle{mytheorem}{.7\linespacing\@plus.3\linespacing}{.7\linespacing\@plus.3\linespacing}%
     {\itshape}
     {}
     {\bfseries}
     {. }
     {0.3ex}
     {\thmname{{\bfseries #1}}\thmnumber{ {\bfseries #2}}\thmnote{ (#3)}}  

\theoremstyle{mytheorem}

\newtheorem{theorem}{Theorem}[section]

\newtheorem{proposition}[theorem]{Proposition}
\newtheorem{corollary}[theorem]{Corollary}
\newtheorem{remark}[theorem]{Remark}
\newtheorem{definition}[theorem]{Definition}
\newtheorem{fact}[theorem]{Fact}

\newcommand{\bbE}{{\ensuremath{\mathbb E}} }

\newcommand{\bbL}{{\ensuremath{\mathbb L}} }

\newcommand{\bbN}{{\ensuremath{\mathbb N}} }

\newcommand{\bbP}{{\ensuremath{\mathbb P}} }

\newcommand{\bbR}{{\ensuremath{\mathbb R}} }

\newcommand{\bbV}{{\ensuremath{\mathbb V}} }

\newcommand{\bbZ}{{\ensuremath{\mathbb Z}} }

\newcommand{\cG}{{\ensuremath{\mathcal G}} }
\newcommand{\cH}{{\ensuremath{\mathcal H}} }

\newcommand{\cW}{{\ensuremath{\mathcal W}} }

\newcommand{\cZ}{{\ensuremath{\mathcal Z}} }

\renewcommand{\tilde}{\widetilde}          
\DeclareMathSymbol{\leqslant}{\mathalpha}{AMSa}{"36} 
\DeclareMathSymbol{\geqslant}{\mathalpha}{AMSa}{"3E} 
\DeclareMathSymbol{\eset}{\mathalpha}{AMSb}{"3F}     

\newcommand{\R}{\mathbb{R}}

\newcommand{\Z}{\mathbb{Z}}
\newcommand{\N}{\mathbb{N}}

\renewcommand{\epsilon}{\varepsilon}
\renewcommand{\theta}{\vartheta}
\renewcommand{\rho}{\varrho}

\newenvironment{myenumerate}{%
\renewcommand{\theenumi}{\arabic{enumi}}%
\renewcommand{\labelenumi}{{\rm(\theenumi)}}%
\begin{list}{\labelenumi}
	{%
	\setlength{\itemsep}{0.4em}%
	\setlength{\topsep}{0.5em}%
	\setlength\leftmargin{2.45em}%
	\setlength\labelwidth{2.05em}%
	\setlength{\labelsep}{0.4em}%
	\usecounter{enumi}%
	}%
	}%
{\end{list}
}

{\end{list}
}

{\end{list}
}

{\end{myenumerate}}

\newenvironment{myitemize}{%
\begin{list}{$\bullet$}%
 	{%
	\setlength{\itemsep}{0.4em}%
	\setlength{\topsep}{0.5em}%
	\setlength\leftmargin{2.45em}%
	\setlength\labelwidth{2.05em}%
	\setlength{\labelsep}{0.4em}%
	}%
	}%
{\end{list}}

\renewenvironment{itemize}{
\begin{myitemize}}%
{\end{myitemize}}

\MakePerPage[2]{footnote} 

\newcommand\independent{\protect\mathpalette{\protect\independenT}{\perp}}
\def\independenT#1#2{\mathrel{\rlap{$#1#2$}\mkern2mu{#1#2}}}

\newcommand{\leqF}{ \leq_{\mathcal{F}}}
\usepackage{bbold} 

\newcommand{\leqLt}{\leq_{\textrm{Lt}}}

\begin{document}
	\begin{abstract} In this note we establish several inequalities and monotonicity properties for the free energy of directed polymers under certain stochastic orders: the usual stochastic order, the Laplace transform order and the convex order. For the latter our results cover also many classical disordered systems.
	\end{abstract}
\title{A  note about Domination and monotonicity in  disordered systems.}
\author[V.L. Nguyen]{Vu-Lan Nguyen }

\maketitle

\section[Introduction]{Introduction}

 In the context of disordered systems, many models are analytically intractable, one could  compare them with more simple models to deduce certain qualitative properties. In this note, we will first focus on the directed polymer model to clarify this idea.
 The directed polymer model was introduced in the statistical physics literature by Huse and Henley \cite{HH85} to study the domain walls of Ising models with impurities and have been applied to many others problems. In recent years they have received much interest because of their strong relationship with the Kardar-Parisi-Zhang (KPZ) equation and  the universality class that it determines (see the review \cite{C12}).

Let $P$ the law of the simple random walk on $\bbZ^d$ with corresponding expectation $E$. Let $\{\omega(i,x): i\in \bbN,\ x\in \bbZ^d \}$ be a collection of real numbers (the environment) and define
\begin{equation*}
Z_N(\omega,\beta) = E\big[e^{\beta\sum_{i=1}^{N}\omega(i,x_i)}\big],
\end{equation*}
 the point-to-line partition function of the directed polymers in environment $\omega$ at inverse temperature $\beta >0$. In this note, in order to be clear, we will denote it also by $Z_n^{pol}$ in some statements (in Section \ref{sec:mtree}). The quantity of interest is the free energy of the system, which is defined as 
\begin{equation*}
p(\beta) = \lim_{N\to\infty} \frac{\log Z_N(\omega,\beta)}{N}.
\end{equation*}
In order to estimate $p(\beta)$, Cook and Derrida  \cite{CD90} introduced the m-tree model as an approximation of the directed polymer. The later model (see Definition \ref{def:mtree}) can be viewed as a directed polymer on a tree  and  can be solved exactly by different alternative approaches, for example by using an analogy with travelling waves or replica approach \cite{CD90},\cite{DS88}. 
This approximation was introduced in mathematics literature  later by Comets and Vargas \cite{CV06} to establish that the free energy $p^{m-tree}(\beta)$ of the m-tree model  is an upper bound for the free energy of directed polymer 
\begin{equation}\label{re:upper_bound}
p(\beta)\leq \inf_{m\geq 1}\frac{1}{m} p^{m-tree}(\beta).
\end{equation}
 The statement (\ref{re:upper_bound}) is quite interesting because it is not trivial that the directed polymer can be obtain as a limit of other models with simpler geometric structure. One natural question is: Can we extend the comparison between these models, for example compare their partition functions, their fluctuations, $\cdots$ under certain stochastic orders ?


{It is well established that stochastic order relations constitute an important tool in the analysis of random variables, mainly applications in various actuarial problems \cite{De01} or comparison of queues \cite{BB03}. Their basic goal  is to compare two distribution functions, not only by comparing the means or the dispersion of these distributions but also expectation of a whole class of functions. } 
%
In general, given two random variables $X$ and $Y$, the binary relation $\leqF$ is defined as follows:
\begin{equation}\label{def:sto_order}
X\leqF Y \qquad \textrm{if} \qquad \mathbb{E}\big[\Phi(X)\big] \leq \mathbb{E}\big[\Phi(Y)\big],
\end{equation}
for all $\Phi\in \mathcal{F}$ such that the integrals are well defined.
The relation $\leqF$ is reflexive and transitive, so it always defines a partial ordering on the set of random variables. Three basic examples of sets $\mathcal{F}$ are:
\begin{itemize}
	\item The set $\{st\}=\{\Phi:\R^n\to\R, \ \Phi\  \textrm{non-decreasing} \}$ generates the increasing  integral order.
	\item The set $\{cx\}=\{\Phi:\R^n\to\R, \ \Phi \ \textrm{convex} \}$ generates the convex integral order.
	\item The set  $\{Lt\}=\{\Phi:\R^+\to\R, \ \Phi(x) = -\exp(-\lambda x) \ \textrm{for all} \ \lambda > 0 \}$ generates the Laplace transform order.
	
\end{itemize}

 The  first main result of this note is that the directed polymer model is  dominated by the m-tree model in the following sense 
\begin{theorem}\label{thm:main}
	Fix $m\geq 1$, and suppose that the random variables $\omega$ have some exponential moments. Then we have the following inequality
	\begin{equation*}
	Z_{n}^{\textrm{pol}}\leqLt Z_{n}^{\textrm{m-tree}}.
	\end{equation*}
	Here $Z_{n}^{\textrm{m-tree}}$ denoted the point-to-line partition function of m-tree model  and $\leqLt$ is the Laplace transform order.
\end{theorem}

As a consequence,  certain estimations, which follow from general considerations for Gaussian environment, can be proved in  general media, such as
\begin{equation*}
\bbE[\log Z_{n}^{\textrm{pol}}]\leq E[\log Z_{n}^{\textrm{m-tree}}] \qquad \textrm{for all} \ n\geq 1.
\end{equation*}

{Moreover, we can deduce an upper bound  (respect to the Laplace transform order)  for the solution of stochastic heat equation (SHE) via the m-tree model.}

The second application of stochastic order for the polymer model is to understand the variability of the free energy as a function of $\beta$. It is folklore that fluctuation of thermodynamic quantities are increasing in the inverse temperature but mathematical formulations of this fact are rare. Let us define the normalized partition function as
\begin{equation*}
W_n(\beta)=E\Big[e^{\beta H_n(x)-n\lambda(\beta)}\Big],
\end{equation*}
where $\lambda(\beta)= \log \bbE(e^{\beta \omega(i,x)} ) <\infty$. Then we have the following  monotonicity property
\begin{theorem}\label{thm:second}
	Fixed $n \in \N$, the process $(W_n(\beta), \beta \in \R_+)$ is increasing in the convex order.
\end{theorem}
Such type of processes are called Peacocks. They  were introduced and intensively studied by {B. Roynette, C. Profeta, F. Hirsch,  and M. Yor \cite{HCY11}}. Actually we will prove that such property is also verified in many disordered systems such as  Sherrington-Kirkpatrick (SK) model, Edwards-Anderson (EA) model, Random field Ising (RFIM) model. 

The note is organised as follows: In Section \ref{sec:pre} we  recall some basic notations of stochastic orders such as the convex order, the Laplace transform order and associated random variables. In Section \ref{sec:mtree}, we  revisit the m-tree model introduced by Cook, Derrida and prove Theorem \ref{thm:main} and its corollaries. In Section \ref{sec:pecock}, we discuss the relationship between the partition functions and Peacocks, and the constructions of   associated martingales. 

\section{Stochastic orders and Correlations}\label{sec:pre}
\subsection{Stochastic orders}
In this section we recall some basic properties of the convex order and the Laplace transform order.
\subsubsection{Usual stochastic order}
Let $X$ and $Y$ be two random variables such that
\begin{equation*}
\bbP[X\leq x] \leq \bbP [Y\leq x] \ \textrm{for all}\ x\in (-\infty,\infty).
\end{equation*}
Then $X$ is said to be smaller than $Y$ in the usual stochastic order, and we denote it by $X\leq_{st} Y$. It is direct to check that this definition is equivalence to the one in Introduction, i.e. by (\ref{def:sto_order}) with $\mathcal{F}=\{st\}$. An other important characterization of the usual stochastic order is its pointwise presentation :
\begin{equation*}
X\leq_{st} Y \iff \exists \ \textrm{ a coupling } \ (\tilde{X}, \tilde{Y}) \ \textrm{of} \ X \ \textrm{and}\ Y \ \textrm{such that} \ \tilde{X} \leq \tilde{Y} \ a.s.
\end{equation*}
This order is useful to compare 2 random variables with different means. Indeed, if $X\leq_{st} Y$ and $\bbE[X]=\bbE[Y]$, then $X=Y$.
\subsubsection{Convex order and peacock}
Let $X$ and $Y$ be two real-valued r.v.'s X is said to be dominated by $Y$ for convex order if, for every convex function $\Phi:\R\to\R$ such that $\bbE[\vert\Phi(X)\vert]<\infty$ and $\bbE[\vert\Phi(Y)\vert]<\infty$, one has:
\begin{equation}
\bbE[\Phi(X)]\leq\bbE[\Phi(Y)].
\end{equation}
This definition is equivalent to say that for every $d\in \R$, we have \cite{SS07}:
\begin{equation*}
\bbE[(d-X)_+]\leq \bbE[(d-Y)_+].
\end{equation*}
The convex order is introduced to compare  the dispersion of random variables  with the same expectation:
\begin{fact}[\cite{SS07}]
	Suppose that $X\leq_{cx} Y$  with $X,Y \in \bbL^1$ then
	\begin{align*}
	\bbE[X]&=\bbE[Y],\\
	\bbV ar(X)&\leq \bbV ar(Y)\\
	\bbE[\vert X-a\vert^p]&\leq\bbE[\vert Y-a\vert^p] \quad \textrm{for all} \ p\geq 1.
	\end{align*}
	
\end{fact}
The latter statement shows that convex order is a fine way to compare the dispersion between two random variables. 
\begin{proof}
	The proof comes directly from the fact that the functions $x \mapsto x $, $x \mapsto -x$ and $x\mapsto x^2$ are convex.
\end{proof}
One of the most interesting features of the convex order is  again here its construction on the same probability space:
\begin{fact}[\cite{Ke72}]\label{Fact:Kellerer}
	Let $X$ and $Y$ be two integrable random variables; $X\leq_{cx} Y$ if and only if there exist two $\R^n$-valued random variables $X_1$ and $Y_1$ defined on the same probability space with  the same distributions as  $X$ and $Y$ respectively, and such that $E[Y_1\vert X_1]=X_1$ a.s.
\end{fact}
By consequence, given a martingale $(X_t,t\geq 0)$ , the previous proposition shows that for every $s<t$, we have:
$$X_s\leq_{cx}X_t.$$
In other words, the one-dimensional marginals of a martingale is an increasing process for the convex order. 
A process $\big(X_t,\ t\geq 0\big)$ is said to be integrable if, for every $t\geq 0$, $\bbE[\vert X_t \vert]<\infty$.
\begin{definition}
	An integrable process which is increasing in the convex order is called a peacock.
\end{definition}
We call a process $\big(X_t,\ t\geq 0\big)$ a $1$-martingale if there exists a martingale $\big(M_t,\ t\geq 0\big)$ such that, for every fixed $t\geq0$:
$$X_t\stackrel{law}{=}M_t.$$
Then Fact \ref{Fact:Kellerer} can be extended to an identity between  peacocks and  $1$-martingales :
\begin{theorem}\cite{Ke72}
	The two following properties are equivalent:
	
	\begin{myenumerate}
		\item $(X_t,\ t \geq 0)$ is a peacock.
		\item $(X_t,\ t \geq 0)$ is a $1$-martingale. 
	\end{myenumerate}

\end{theorem}
This theorem shows that a peacock can be associated to a martingale which shares the same one-dimensional marginals with it. But normally  it is  a non-trivial question to find such embedding which would be natural in the context. In Section 4, the martingale representations of partition functions will be provided in certain special cases. One can check the book  \cite{HCY11} to get more information about various methods to construct the associated martingales. 

\subsubsection{Laplace transform order}

Given two  random variables $X$ and $Y$, $X$ is said to be smaller than $Y$ in the Laplace transform order, denoted as $X\leqLt Y $, when the inequality
\begin{equation}\label{def:Lt}
\bbE[-\exp(-\lambda X)]\leq \bbE[-\exp(-\lambda Y)] \quad \textrm{holds for all} \ \lambda\geq 0.
\end{equation}
By rewriting the definition \ref{def:Lt}, on can easily deduce that
\begin{equation*}
X\leqLt Y\iff \frac{X}{\epsilon}\leq_{st} \frac{Y}{\epsilon'},
\end{equation*}
where $\epsilon\independent X, \ \epsilon\sim Exp(1)$ , $\epsilon'\independent Y, \ \epsilon'\sim Exp(1)$.
Let $U_{cm}$ be the class of the completely monotone functions, that is the class of the functions $\Phi:(0,+\infty)\to \R^+$ satisfying
$$(-1)^k\Phi^{(k)}\geq 0 \quad  \textrm{for all} \quad  k\geq 0,$$
where $\Phi^{(k)}$ denotes the $k$th derivative of $\Phi$. An example of such function is $x\mapsto e^{-\lambda x}$. It is well known that for $\Phi\in U_{cm}$, there exists a positive measure $\mu$ on $\R^+$, not necessarily finite, such that
\begin{equation*}
\Phi(x)=\int_{0}^{+\infty} \exp(-tx)d\mu(t), \quad x\in \R^+,
\end{equation*} 
see for example  Theorem 1a, page 416 \cite{Fe66}. Then the Laplace transform order and completely monotone functions are related as follows \cite{Lin98}
\begin{proposition}\label{theo:Lin}:
	Let $X$ and $Y$ be two positive random variables; then $X\leqLt Y$ is equivalent to each one of the following statements:
	\begin{myenumerate}
		\item $\bbE f(X)\geq \bbE f(Y)$ for each c.m function $f:(0,+\infty)\to [0,+\infty)$.
		\item $\bbE f(X)\leq \bbE f(Y)$ for each function $f:(0,+\infty)\to (-\infty,+\infty)$ having a c.m derivative, provided the expectation exist.
		\item $\bbE f(X)\leq \bbE f(Y)$ for each  function $f:(0,+\infty)\to [0,+\infty)$ having a c.m derivative.
	\end{myenumerate}
\end{proposition}
\begin{remark}\label{Remark: Lt}
	The function $f(x)=\log x $ and $f(x)=x^{\alpha}$ with $\alpha \in (0,1]$ have  c.m. derivatives. Then, $X\leqLt Y$ implies
	\begin{align*}
	\bbE[\log X]&\leq \bbE[\log Y],\\
	\bbE[ X^{\alpha}]&\leq \bbE[Y^{\alpha}].
	\end{align*}
 Moreover it is also true that:
 \begin{equation*}
 \log X\leqLt \log Y.
 \end{equation*}	
 The proof of this fact will be given later in page 12.
	
\end{remark} 

\subsection{Correlation and stochastic order}
One of the main difficulties in studying directed polymers comes from its complicated correlation structure. In this section, we would like to recall some useful tools to compare random vectors and in our case different disordered systems. In the Gaussian case, the main tool is the Slepian's lemma \cite{Ka86}.
\begin{theorem}[Slepian's lemma]\label{lem:Slepian}
	Let $(X_i)_{i\leq n}$, $(Y_i)_{i\leq n}$ two gaussian vectors with mean $0$, and $F:\R^n\to \R$  all of whose partial derivatives up to second order have subgaussian growth. If
	\begin{equation*}
	E(X_i^2)=E(Y_i^2), \qquad \qquad E(X_iX_j)\leq E(Y_iY_j),
	\end{equation*}
	and
	$$\frac{\partial^2 F}{\partial x_i \partial x_j}\geq 0,$$
	for all $i \ne j$, then
	\begin{equation*}
	EF(X_1,\ldots,X_n)\leq EF(Y_1,\ldots,Y_n).
	\end{equation*}
\end{theorem}
This result is a powerful and elegant way to compare the observables of two different Gaussian vectors. In particular in order to compare the maximum of two vectors \cite{A90}:
\begin{corollary}\label{Cor: Slepian inequlity}
	Let $X,Y$ be centered Gaussian vectors. Assume that $\bbE [X_i^2]=\bbE[ Y_i^2]$ and $\bbE [X_iX_j] \geq \bbE [Y_iY_j]$ for all $i\ne j$. Then $\max_i X_i$ is stochastically dominated by $\max_i Y_i$, i.e.
	\begin{equation*}
	\max_i X_i \leq_{st} \max_i Y_i.
	\end{equation*}
\end{corollary}
In Section 3, we will provide an application of Slepian's lemma (in the case of  Gaussian environment) to reprove a classic result in \cite{CV06}. 
For general type of distributions, one can try to estimate an observable of a random vector by using the notion of association.  
\begin{definition}
	A $\R^n$ valued random vector $X=(X_1,X_2,\ldots,X_n)$ is called  associated if the inequality  
	\begin{equation*}
	\bbE[f(X)g(X)]\geq \bbE[f(X)]\bbE[g(X)],
	\end{equation*}
	holds for all coordinatewise non-decreasing (or non-increasing) mappings $f,g:\R^n\to \R$, for which the involved expectations exist. 
\end{definition}

The real vector $\{\overline{X}_1,\ldots,\overline{X}_n\}$ is called an independent version of the random vector $\{X_1,..,X_n\}$ if the  variables $\{\overline{X}_1,\ldots,\overline{X}_n\}$ are mutually independent, and for all $1\leq i\leq n$, the coordinates  $X_i$ and $\overline{X}_i$ have the same  distribution. If X is associated, then an immediate induction shows that for all $n$-tuples of non-negative non-decreasing (or non-increasing) $f_i:\R\to\R_+$, then
\begin{equation}
\label{lem:associated}
\bbE\Big[\prod_{i=1}^{n}f_i(X_i)\Big]\geq \prod_{i=1}^{n}\bbE\Big[f_i(X_i)\Big]=\bbE\Big[\prod_{i=1}^{n}f_i(\overline{X}_i)\Big].
\end{equation}
{We stress on the requirement that $f_i\geq 0\ \textrm{for all}\  i$.}
In general, computing explicitly the joint distribution of a vector $(X_i)_{i\leq n}$ is not trivial task, but the association property can often be established using the following properties \cite{BB03} :
\begin{proposition}\label{pro: asso}
	\  
	\begin{itemize}
		\item The union of independent sets of associated random variables forms a set of associated random variables.
		\item For all non-decreasing function $\phi:\R^n\to\R$, and all set of associated random variables $\{X_1,..,X_n\},$ the set of random variables 
		\begin{equation*}
		\{\phi(X_1,...,X_n),X_1,..,X_n\} 
		\end{equation*}
		is associated.
	\end{itemize}
\end{proposition}

\section{m-tree domination}\label{sec:mtree}
\subsection{m-tree model}

Let $(\textbf{x}=(x_n)_{n\in\N},(P^x)_{x\in\Z^d})$ denote the simple random walk on the $d$-dimension integer lattice $\Z^d$. We denote $(\omega(n,x))_{(n,x)\in \N\times\Z^d }$ the random environment on $(\Omega,\cG, \bbP)$.We suppose from now on that the exponential moment of $\omega$ exists at least for certain small $\beta >0$.
We define for $k<n$,
\begin{equation*}
H_{k,n}(\textbf{x})= \sum_{j=1}^{n-k}\omega(k+j,x_j),
\end{equation*}
and the point-to-point partition function is given by
\begin{equation*}
Z_{k,n}^{x_0}(y)=P^{x_0}\Big(e^{\beta H_{k,n}(\textbf{x})} 1_{ x_{n-k}=y}\Big),
\end{equation*}
and the point-to-line partition function is given by 
\begin{equation*}
Z_{k,n}^{x_0}=P^{x_0}\Big(e^{\beta H_{k,n}(\textbf{x})} \Big).
\end{equation*}
If $x_0=0$ and $k=0$, we will denote shortly as $Z_n$ for the point-to-line partition functions, and $Z_n(y)$ for the point-to-point partition functions. Moreover, we obtain the following recursion for $ Z_{k,n}^{x_0}$
\begin{equation}\label{equ:rec}
Z_{k,n}^{x_0}=\frac{1}{2d}\sum_{i=1}^{d} e^{\beta \omega_{k+1,x_0+e_i}}  Z_{k+1,n}^{x_0+e_i}.
\end{equation}
where $\{e_i, i\leq d\}$ is the canonical basis of $Z^d$. 
This formula is quite simple but hard to analyse because, for every $i,j$, the variables $Z_{k+1,n}^{x_0+e_i}$  and $Z_{k+1,n}^{x_0+e_j}$ are correlated. To obtain the m-tree approximation to this problem, one can iterate (\ref{equ:rec})  $m$ times exactly and then neglects the remaining correlations. In other words one takes account of the correlations on first $m$ steps of the lattice and then constructs a tree from this pattern, by neglecting other correlations. The construction of a 3-tree for directed polymer in dimension $1+1$ is illustrated in  Figure 1. 
\begin{figure}\label{Fig:3-tree}
	\begin{center}

		\begin{tikzpicture}[scale=0.3]
		\draw[-,thick, red](-15,-15)--(0,0)--(15,-15);
		\draw[-,thick,red](-4,-4)--(7,-15);
		\draw[-,thick,red](-4,-4)--(7,-15);
		\draw[-,thick,red](-8,-8)--(-1,-15); 
		\draw[-,thick,red](-12,-12)--(-9,-15);
		\draw[-,thick,red](-13,-13)--(-11,-15)--(-10,-14);
		\draw[-,thick,red](-14,-14)--(-13,-15)--(-11,-13);
		\draw[-,thick,red](4,-4)--(-7,-15);
		\draw[-,thick,red](8,-8)--(1,-15);
		\draw[-,thick,red](12,-12)--(9,-15);  	
		\draw[-,thick,red](-5,-13)--(-3,-15)--(-2,-14);
		\draw[-,thick,red](-6,-14)--(-5,-15)--(-3,-13);
		\draw[-,thick,red](3,-13)--(5,-15)--(6,-14);
		\draw[-,thick,red](2,-14)--(3,-15)--(5,-13);
		\draw[-,thick,red](11,-13)--(13,-15)--(14,-14);
		\draw[-,thick,red](10,-14)--(11,-15)--(13,-13);
		\draw[thick, blue](-12,-12)--(12,-12);
		\draw[thick, blue](-15,-15)--(15,-15);
		\node[blue] at (-15,-12){{\bf $m=3$}};
		\node[blue] at (-18,-15){{\bf $2m=6$}};
		
		\fill[red] (-15,-15)--(-16,-16)--(-14,-16)--(-15,-15);
		\fill[red] (-13,-15)--(-14,-16)--(-12,-16)--(-13,-15);
		\fill[red] (-11,-15)--(-12,-16)--(-10,-16)--(-11,-15);
		\fill[red] (-9,-15)--(-10,-16)--(-8,-16)--(-9,-15);
		\fill[red] (-7,-15)--(-8,-16)--(-6,-16)--(-7,-15);
		\fill[red] (-5,-15)--(-6,-16)--(-4,-16)--(-5,-15);
		\fill[red] (-3,-15)--(-4,-16)--(-2,-16)--(-3,-15);             
		\fill[red] (-1,-15)--(-2,-16)--(0,-16)--(-1,-15);     
		\fill[red] (1,-15)--(0,-16)--(2,-16)--(1,-15);     
		\fill[red] (3,-15)--(2,-16)--(4,-16)--(3,-15);     
		\fill[red] (5,-15)--(4,-16)--(6,-16)--(5,-15);
		\fill[red] (7,-15)--(6,-16)--(8,-16)--(7,-15);
		\fill[red] (9,-15)--(8,-16)--(10,-16)--(9,-15);    
		\fill[red] (11,-15)--(10,-16)--(12,-16)--(11,-15);    
		\fill[red] (13,-15)--(12,-16)--(14,-16)--(13,-15);
		\fill[red] (15,-15)--(14,-16)--(16,-16)--(15,-15);
		\end{tikzpicture}
	\end{center}
	\caption{3-tree drawn for d=2}
\end{figure}
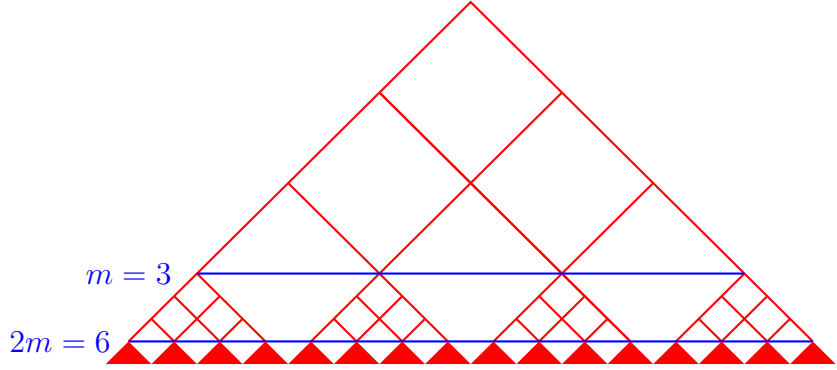
Mathematically, the m-tree can be defined by using the multiplicative cascades formalism \cite{RV13}\cite{Liu00}, that we introduce first. 

Let $N\geq 2$ be a fixed integer and $[\![1,N]\!]=\{1,2,\cdots,N\}$
\begin{equation*}
U=\bigcup_{k\in \N}[\![1,N]\!]^k,
\end{equation*}
be the set of all finite words $u=u_1u_2...u_k$ of elements in the alphabet $\{1,2,...,N-1,N\}$. For  $u=u_1u_2...u_k$ and  $v=v_1v_2...v_k$ two finite words, let $uv$ denote the words $u_1...u_kv_1...v_k$. Given a probability distribution $q$ on $\R_+^N$, it is known that there exist a probability space with probability measure denoted by $\bbP$ and random variables $(A_u)_{u\in U}$ defined on this space, such that the random vectors $(A_{u1},...,A_{uN})_{u\in U}$ form an i.i.d sequence with common distribution $q$.
\begin{definition}
	We define the cascade $W_n^{\textrm{casc}}$ by
	\begin{equation}\label{for:tree_partition}
	W_n^{\textrm{casc}}=\sum_{u_1,..,u_n\in[\![1,N]\!]} A_{u_1}A_{u_1u_2}...A_{u_1..u_n}.
	\end{equation}
	Let denote the filtration by 
	\begin{equation*}
	\mathcal{G}_n:= \sigma\{A_u; \vert u\vert\leq n\},
	\end{equation*}	
\end{definition}
If $\sum_{i=1}^{N} \bbE A_{u_i}=1$ then  $(W_n^{\textrm{casc}},\mathcal{G}_n)$ is a non negative martingale. The m-tree corresponds to a particular choice of $q$. Let $m\geq 1$ and define $L_m$ to be the set of points visited by the simple random walk at time $m$:
\begin{equation}\label{def:Lm}
L_m=\{x\in\Z^d; P(x_m=x)>0\}.
\end{equation}
\begin{definition}\label{def:mtree}
	The partition function $(Z_n^{m-tree})$, $n\in m\mathbb{N}$, of the m-tree is the multiplicative cascade associated with the alphabet $L_m$ and $q=q_n$ given by  the joint law of point-to point partition functions , $(Z_m(x))_{x\in L_m}$ at time $m$.
	In particular $Z_n^{m-tree}\stackrel{\mathcal{D}}{=}Z_m^{pol}$.
\end{definition}
\begin{remark}

	By the same construction, for a fixed endpoint $x\in \mathbb{Z}^d$, we can also define the point-to-point partition function for the m-tree model and we denote it
	$Z^{\textrm{m-tree}}_{n}(x)$. More precisely, we sum over all the paths from $(0,0)$ to $x$, and the contribution of each path depends on when it went through at time $m, 2m,\cdots, lm$ as in the description of the multiplicative cascade. \medskip
\end{remark}

It is well known from the multiplicative cascade theory \cite{Liu00} that the free energy $p^{m-tree}$ exists,
$$p^{m-tree}(\beta)=\lim_{l\to \infty}\frac{\log Z_{ml}^{m-tree}}{l}.$$
Moreover it gives  an upper bound for the free energy of directed polymer (Theorem 3.3 in \cite{CV06}):
\begin{equation}\label{re:upper}
p(\beta)\leq \inf_{m\geq 1}\frac{1}{m} p^{m-tree}(\beta).
\end{equation}
\subsection{Gaussian case}
In the case of Gaussian environment, we can recover  the result (\ref{re:upper}) as a direct consequence of the following stronger theorem:
\begin{theorem}\label{re:gaussian}
	Given $m>0$, for $n=ml$, in the case of Gaussian environment we have the following inequality
	\begin{equation*}
	\bbE[\log Z_{n}^{\textrm{pol}}]\leq E[\log Z_{n}^{\textrm{m-tree}}].
	\end{equation*}
	
\end{theorem}
\begin{proof}
	By the Gaussian assumption, we want to apply Slepian's lemma (\ref{lem:Slepian}). Define $\Pi_n$ the set of polymer paths $\bf x$ on $\Z^d$ of length $n$ and the two Gaussian vectors are $\big(H_n(\textbf{x}), \textbf{x}\in \Pi_n\big)$ and $\big(H_{n}^{m-tree}(\textbf{x}), \textbf{x}\in \Pi_n\big)$. The function $F$ is defined by
	$$F(h)=\log\sum_{\textbf{x}\in\Pi_n}\exp\{\beta h(\textbf{x})\}\qquad \textrm{for}\quad h=(h(\textbf{x});\textbf{x}).$$
	It is clear that $F\big((H_n(\textbf{x});\textbf{x})\big)=\log Z_{n}^{\textrm{pol}}$ and $F\big((H_{n}^{m-tree}(\textbf{x}); \textbf{x}\big)=\log Z_{n}^{\textrm{m-tree}}$, for $n=lm$.
	Moreover, if $\textbf{x},\textbf{x}'\in \Pi_n$, we have for $\textbf{x}\ne \textbf{x}'$,
	$$\frac{\partial^2 F}{\partial h(\textbf{x}) \partial h(\textbf{x}')}(h)=\beta^2 Cov_{Q_{n,\beta}^h} \big(1_{\{\textbf{x}\}},1_{\{\textbf{x}'\}}\big)= -\frac{\beta^2 \exp(h(\textbf{x})+h(\textbf{x}'))}{\big(\sum_{\textbf{x}\in\Pi_n}\exp\{\beta h(\textbf{x})\}\big)^2}\leq 0$$
	where the probability measure $Q_{n,\beta}^h$ is the one with Hamiltonian $h=(h(\textbf{x});\textbf{x})$.
	On the other hand, we can check that the energy $H_n(\textbf{x})$ is a Gaussian vector with mean $0$ and covariance
	$$\bbE\big[H_n(\textbf{x})H_n(\textbf{x}')\big]= \sum_{i=1}^{n} 1_{\{\textbf{x}_i=\textbf{x}'_i\}},$$
	and the covariance in case of the  m-tree $H_{n}^{m-tree}(\textbf{x})$ is given by
	\begin{align*}
	\bbE\big[(H_{n}^{m-tree}(\textbf{x})(H_{n}^{m-tree}(\textbf{x}')\big]&=\sum_{0\leq l\leq n/m} \Bigg(\Big[\prod_{k=0}^{l}1_{\{\textbf{x}_{km}=\textbf{x}'_{km}\}}\Big]\times \sum_{i=lm}^{((l+1)m-1)\wedge n}1_{\{\textbf{x}_i=\textbf{x}'_i\}}\Bigg).\\
	&= \sum_{i=1}^{n} \mathbb{1}_{\textbf{x}_i=\textbf{x}'_i} \prod_{k=1}^{i/m} \mathbb{1}_{\textbf{x}_{km}=\textbf{x}'_{km}}
	\end{align*}
	One can check that
	$$\bbE\big[H_n(\textbf{x})H_n(\textbf{x}')\big]\geq \bbE\big[(H_{n}^{m-tree}(\textbf{x})(H_{n}^{m-tree}(\textbf{x}')\big],$$
	and 
	\begin{equation*}
	\bbE[H_n(\textbf{x})^2]=\bbE[H_{n}^{m-tree}(\textbf{x})^2]=n.
	\end{equation*}
	
	Now Slepian's lemma (Theorem \ref{lem:Slepian}) yields the result.
\end{proof}

Moreover, by using the Slepian's inequality (Corollary \ref{Cor: Slepian inequlity} ) we obtain 
\begin{equation}\label{conj:1}
\sup_{\textbf{x}} H_n^{pol}(\textbf{x})\leq_{st} \sup_{\textbf{x}} H_n^{m-tree}(\textbf{x}).
\end{equation}
Recall that the supremum is the free energy at zero temperature. One could hope that the corresponding statement could be true at positive temperature:
\begin{equation}\label{eq:conj}
\textrm{ Is it true that:}\qquad \log Z_n^{pol}\leq_{st} \log Z_n^{m-tree}\quad ?
\end{equation}
Actually if it was true, then it would imply that :
\begin{equation}\label{eq: st}
Z_n^{pol}\leq_{st}  Z_n^{m-tree}.
\end{equation}
Since $\bbE Z_n^{pol}= \bbE Z_n^{m-tree}$, (\ref{eq: st}) implies that $Z_n^{pol} \stackrel{\mathcal{D}}{=} Z_n^{m-tree}$, which is clearly wrong. 
In the next section we will prove that the inequality (\ref{eq:conj}) is  true with respect to the Laplace transform order instead of the usual stochastic order.

\subsection{Laplace order domination with general case}
 Since the stochastic domination \ref{conj:1} does not extend to positive temperature, it is natural to look for a different order for which the extension holds, at all temperature and with any disorder. 
 In the rest of this section we will prove that Theorem \ref{re:gaussian}  can be formulated within a general framework of a standard stochastic ordering and provide some direct consequences, specially an upper bound for the solution of the Stochastic Heat Equation in dimension 1.
\begin{theorem}\label{thm: main}
	Given $m\geq 1$, we have the following inequality:
	\begin{equation*}
	Z_{n}^{\textrm{pol}}\leqLt Z_{n}^{\textrm{m-tree}},
	\end{equation*}
	and for a fixed endpoint $x$:
	\begin{equation*}
	Z_{n}^{\textrm{pol}}(x)\leqLt Z_{n}^{\textrm{m-tree}}(x).
	\end{equation*}
\end{theorem}

Before proving this theorem, let us derive some direct consequences from the Laplace transform order properties: 
\begin{corollary}\label{Col:main}
	We have for all  $\alpha\in(0,1)$ :
	\begin{align*}
	\log Z_n^{pol}&\leq_{Lt} \log Z_n^{m-tree},\\
	\bbE[\log Z_{n}^{\textrm{pol}}]&\leq	\bbE[\log Z_{n}^{\textrm{m-tree}}],\\
	\bbE[ \big(Z_{n}^{\textrm{pol}}\big)^{\alpha}]&\leq\bbE[\big( Z_{n}^{\textrm{m-tree}}\big)^{\alpha}] ,\\
	\bbE\big[Z_{n}^{\textrm{pol}} \log Z_{n}^{\textrm{pol}}\big]&\geq\bbE\big[Z_{n}^{\textrm{m-tree}} \log Z_{n}^{\textrm{m-tree}}\big].
	\end{align*}
\end{corollary}

\begin{proof}[Proof of Theorem \ref{thm: main}]
	
	We will present here only the proof for the case of point-to-line partitions function. The proof for the point-to-point case is similar. The proof is based on the simple fact that given a set of  pair of times $\{(m_1,n_1),\ldots,(m_k,n_k)\}$  then the variables  $\{Z_{m_1,n_1}^{x_1}(x_1'),\ldots,Z_{m_k,n_k}^{x_k}(x_k')\}$ are positively associated ( as can be seen from Proposition \ref{pro: asso}). Here $Z_{m_1,n_1}^{x_1}(x_1')$ refers to the partition function of lattice polymer model, as well as $Z_n^{pol}$. Let us fix $m\in\N$ and  consider the point-to-line partition function $Z_{km}$. Recall that the notation of $L_m$ is given in (\ref{def:Lm}).
	We will prove by induction on $h\geq 1$ that there exists a set of random variables $$\Big(\big(\overline{Z}_{im,(i+1)m}^x(x')\big)_{(x,x')\in L_{im}\times L_{(i+1)m}}\Big)_{k-h\leq i\leq k},$$
	such that the random vectors $$\big(\overline{Z}_{im,(i+1)m}^x(x')\big)_{(x,x')\in L_{im}\times L_{(i+1)m}}$$ with $k-h\leq i\leq k$ are mutually independent and have the same law as $\big(\overline{Z}_{0,m}(x')\big)_{x'\in L_{m}}$ and satisfies the following inequality:
	\begin{align*}
	\bbE\big[\exp(-\lambda Z_{km}^{pol})\big]\geq \bbE\Bigg[\prod_{(x_{k-h},\ldots,x_{k-1})\in L_{(k-h)m}\times\ldots\times L_{(k-1)m}} \exp\Big(-\lambda &Z_{0,(k-h)m}(x_{k-h})\ldots\\\overline{Z}_{(k-h)m,(k-h+1)m}^{x_{k-h}}(x_{k-h+1})
	&\overline{Z}_{(k-1)m,km}^{x_{k-1}}\Big)\Bigg].
	\end{align*}
	It clear that that if $h=k-1$ then the right hand side is the partition function of the $m$-tree, and it yields the result. 
	
	Let us begin with the case $h=1$. By Markov property, the partition function $Z_{km}$ can be represented as :
	
	\begin{equation*}
	Z_{km}^{pol}=\sum_{x\in L_{(k-1)m}} Z_{0,(n-1)m}(x)Z_{(k-1)m,km}^x.
	\end{equation*}
	Fixed $\lambda>0$, we have 
	\begin{equation*}
	\bbE\big[\exp(-\lambda Z_{km}^{pol})\big]= \bbE\Big[\prod_{x\in L_{(k-1)m}} \exp\big(-\lambda Z_{0,(n-1)m}(x)Z_{(k-1)m,km}^x\big)\Big].
	\end{equation*}
	Let us denote $G_{k-1}=\sigma\{\omega(i,x), i\leq (k-1)m \}$, then
	\begin{equation*}
	\bbE\big[\exp(-\lambda Z_{km}^{pol})\big]= \bbE\Big[\bbE\Big[\prod_{x\in L_{(k-1)m}} \exp\big(-\lambda Z_{0,(n-1)m}(x)Z_{(k-1)m,km}^x\big)\vert G_{k-1}\Big]\Big].
	\end{equation*}
	Under the conditional probability $\bbP[.\vert G_{k-1}]$, the  variables $\{\omega(i,x),\ i\geq (k-1)m+1\}$ are still independent, hence $\{Z_{(k-1)m,km}^x, x\in \bbZ^d\}$ are associated. By (\ref{lem:associated}), we obtain:
	\begin{equation*}
	\bbE[\exp(-\lambda Z_{km}^{pol})]\geq \bbE\Big[\bbE\Big[\prod_{x\in L_{(k-1)m}} \exp(-\lambda Z_{0,(n-1)m}(x)\overline{Z}_{(k-1)m,km}^x)\vert G_{k-1}\Big]\Big],
	\end{equation*}
	where the random variables $\overline{Z}_{(k-1)m,km}^x$ are mutually independent and do not depend on $G_{k-1}$ as well.
	By consequence
	\begin{equation*}
	\bbE[\exp(-\lambda Z_{km}^{pol})]\geq \bbE\Big[\prod_{x\in L_{(k-1)m}} \exp(-\lambda Z_{0,(n-1)m}(x)\overline{Z}_{(k-1)m,km}^x)\Big].
	\end{equation*}
	Now, we suppose that the result holds for $h\geq 1$ , 
	\begin{align*}
	\bbE[\exp(-\lambda Z_{km}^{pol})]\geq \bbE\Big[&\prod_{(x_{k-h},\ldots,x_{k-1})\in L_{(k-h)m}\times\ldots\times L_{(k-1)m}} \exp(-\lambda Z_{0,k-h}(x_{k-h})\\ &\overline{Z}_{(k-h)m,(k-h+1)m}^{x_{k-h}}(x_{k-h+1})\cdots\overline{Z}_{(k-1)m,km}^{x_{k-1}})\Big],
	\end{align*}
	we will prove it holds for $h+1$ as well. 
	For $x_{k-h}\in L_{(k-h)m}$, we have 
	\begin{equation*}
	Z_{0,k-h}(x_{k-h}) =\sum_{x_{k-h-1}\in L_{(k-h-1)m}}Z_{0,(k-h-1)m}(x_{k-h-1})Z_{(k-h-1)m,(k-h)m}^{x_{k-h-1}}(x_{k-h}).
	\end{equation*}
	By using again (\ref{lem:associated}), we can insert the independent structure in such a way that for all $x_{k-h-1}\in L_{(k-h-1)m}$ then the set of random variables $(\overline{Z}_{(k-h-1)m,(k-h)m}^{x_{k-h-1}}(x_{k-h}))_{x_{k-h}\in  L_{(k-h)m}}$ are mutually independent and have the same law. By consequence, we obtain: 
	
	\begin{align*}
	\bbE\big[\exp(-\lambda Z_{km}^{pol})\big]\geq \bbE\Big[\prod_{(x_{k-h-1},\ldots,x_{k-1})\in L_{(k-h-1)m}\times..\times L_{(k-1)m}} \exp\big(-\lambda &Z_{0,k-h-1}(x_{k-h-1})\ldots\\\overline{Z}_{(k-h-1)m,(k-h)m}^{x_{k-h-1}}(x_{k-h})&\overline{Z}_{(k-1)m,km}^{x_{k-1}}\big)\Big],
	\end{align*}
	and this yields the result.
\end{proof}

\begin{proof}[Proof of Corollary \ref{Col:main}]
	The first result is a direct consequence of Theorem \ref{thm: main}. Indeed since the function $x\mapsto x^{-\lambda}$, with $\lambda>0$, is a completely monotone function and $Z_n^{pol}\leq_{Lt} Z_n^{m-tree}$, we obtain:
	\begin{equation*}
	\mathbb{E}\big[(Z_n^{pol})^{-\lambda}\big]\geq \mathbb{E}\big[(Z_n^{m-tree})^{-\lambda}\big], \qquad \textrm{for}\ \lambda >0.
	\end{equation*}
	Hence
	\begin{equation*}
	\mathbb{E}\Big[\exp\big(-\lambda \log Z_n^{pol}\big)\Big]\geq\mathbb{E}\Big[\exp\big(-\lambda \log Z_n^{m-tree}\big)\Big], \qquad \textrm{for}\ \lambda >0,
	\end{equation*}
	yields that $$\log Z_n^{pol}\leqLt \log Z_n^{m-tree}.$$ 
	The second and third results are a directly consequence of Remark \ref{Remark: Lt}. For the last one, by using the fact that $\bbE[Z_{n}^{\textrm{pol}}]=\bbE[Z_{n}^{\textrm{m-tree}}]$, for $0\leq \alpha<1$ we have
	\begin{equation*}
	\bbE\Big[\frac{\big(Z_{n}^{\textrm{pol}}\big)^{\alpha}-Z_{n}^{\textrm{pol}}}{\alpha-1}\Big]\geq \bbE\Big[\frac{\big(Z_{n}^{\textrm{m-tree}}\big)^{\alpha}-Z_{n}^{\textrm{m-tree}}}{\alpha-1}\Big].
	\end{equation*}
	By taking the limit $\alpha \to 1$, we get the result.
\end{proof}
As a consequence, we can obtain an upper bound for the solution of Stochastic Heat Equation (SHE) . The SHE is defined shortly as 
\begin{equation}
\partial \cZ = \frac{1}{2} \Delta \cZ +\cZ \cW,
\end{equation}
with the initial condition $\cZ(0,x)= \delta_0(x)$ and $\cW$ is a space-time white noise. This equation is well-posed and is related to the  Kardar-Parisi-Zhang (KPZ) equation equation, via the transformation $\cH=\log \cZ$. The following theorem by Alberts-Khanin-Quastel (\cite{AKQ14}) shows that the solution of SHE can be obtained as the limit of the renormalized partition function of discrete directed polymers.
\begin{theorem}
	For each $t>0$ and $x\in \bbR$, we have the convergence in law 
	\begin{equation}\label{pro:conv}
	\cZ_n(t,x) \to \sqrt{4\pi} e^{\frac{x^2}{4t}} \cZ(2t,x),
	\end{equation}
	where
	\begin{equation*}
	\cZ_n(t,x) := \frac{Z_{tn}(x\sqrt{n},\beta_n=n^{-1/4})}{\bbE Z_{tn}(x\sqrt{n},\beta_n=n^{-1/4})},
	\end{equation*}
	and $Z_{tn}(x\sqrt{n},\beta_n=n^{-1/4})$ is the point-to-point partition function at $x\sqrt{n}$ with the  temperature $n^{-1/4}$.
\end{theorem}
\textbf{Remark.} To keep the notation simple, we denote the endpoint by $x\sqrt{n}$ by $\lfloor x\sqrt{n} \rfloor$, though it is understood as a lattice point.

By Theorem \ref{thm: main}, we can establish that for all endpoint $x\sqrt{n}$:
\begin{equation}\label{ineq:p2p}
Z_{tn}(x\sqrt{n},\beta_n=n^{-1/4})\leqLt Z^{\textrm{m-tree}}_{tn}(x\sqrt{n}, n ^{-1/4}),
\end{equation}
which implies 
\begin{equation*}
\frac{Z_{tn}(x\sqrt{n},\beta_n=n^{-1/4})}{\bbE Z_{tn}(x\sqrt{n},\beta_n=n^{-1/4})}\leqLt \frac{Z^{\textrm{m-tree}}_{tn}(x\sqrt{n}, n ^{-1/4})}{\bbE Z^{\textrm{m-tree}}_{tn}(x\sqrt{n}, n ^{-1/4})}.
\end{equation*}
By taking $n\to \infty$, the left hand side converges to $\sqrt{4\pi} e^{\frac{x^2}{4t}} \cZ(2t,x)$ as in (\ref{pro:conv}). For the right hand side, we use   Fatou's lemma and easily get the following result:
\begin{theorem}
	We have
	\begin{equation*}
	\cZ(2t,x)\leqLt \frac{1}{\sqrt{4\pi} e^{\frac{x^2}{4t}}}\limsup_{n\to\infty} \frac{Z^{\textrm{m-tree}}_{tn}(x\sqrt{n}, n ^{-1/4})}{\bbE Z^{\textrm{m-tree}}_{tn}(x\sqrt{n}, n ^{-1/4})}.
	\end{equation*}
\end{theorem}
\begin{remark}
	The limit of the right hand side should exist as an universal object corresponding to the intermediate limit of the m-tree, this has been verified in the case of directed polymer on square lattice \cite{AKQ14} or on a hierarchical lattice \cite{ACK15}. This interesting question will not be studied in the scope of this note but in a future research. 
\end{remark}
\section{Disorder systems and Peacock}\label{sec:pecock}
In this section we will consider many disordered systems, but in a first time, in order to fix the idea let us start with the directed polymer. Recall the normalized partition function is given by
\begin{equation*}
W_n(\beta)=E\Big[e^{\beta H_n(x)-n\lambda(\beta)}\Big].
\end{equation*}
\begin{proof}[Proof of Theorem \ref{thm:second}]
	Let $\Phi:\R\to\R$ a convex function. We need to prove that the function $\beta\mapsto \bbE\big[\Phi(W_n(\beta))\big]$ is increasing in $\beta$. We have
	\begin{equation*}
	\frac{d}{d\beta} \bbE\big[\Phi(W_n(\beta))\big]=\bbE E\Big[\Phi'(W_n(\beta))\big(H_n(x)-n\lambda'(\beta)\big)e^{\beta H_n(x)-n\lambda(\beta)}\Big].
	\end{equation*}
	By Fubini, the right hand side is equal to:
	$$E\ \bbE\Big[\big(H_n(x)-n\lambda'(\beta)\big)e^{\beta H_n(x)-n\lambda(\beta)}\Phi'(W_n(\beta))\Big].$$
	For any fixed path $x$,we define the probability measure $Q^x$,
	\begin{equation*}
	\frac{dQ^{x}}{d\bbP}=e^{\beta H_n(x)-n\lambda(\beta)}.
	\end{equation*}
	Under $Q^x$, the random environment $\omega(n,x)$ are still independent, then are positively associated and these two following applications are increasing for $\beta \geq 0$:
	\begin{align*}
	(\omega)&\mapsto H_n(x,\omega),\\
	(\omega)&\mapsto \Phi'(W_n(\beta,\omega)).
	\end{align*} 
	Here we use the fact that the function $\Phi'(x)$ is increasing. By using the property of positively associated random variables we get :
	\begin{align*}
	\frac{d}{d\beta} \bbE\big[\Phi(W_n(\beta))\big]&=E\ Q^x\Big[ \big(H_n(x)-n\lambda'(\beta)\big)\Phi'(W_n(\beta))\Big]\\
	&\geq 
	E\ Q^x\Big[H_n(x)-n\lambda'(\beta)\Big]
	E\ Q^x\Big[ \Phi'(W_n(\beta))\Big].
	\end{align*} 
	But by a simple calculation we can verify that:
	$$E\ Q^x\Big[H_n(x)-n\lambda'(\beta)\Big]=0,$$
	that means
	$$\frac{d}{d\beta} \bbE\big[\Phi(W_n(\beta))\big]\geq 0,$$
	which yields the result.
\end{proof}
This result can be generalized in the scope of classical spin glass model such as
\begin{itemize}
	\item Sherrington-Kirkpatrick (SK) model
	\item Edwards-Anderson (EA) model
	\item Random field Ising (RFIM) model
\end{itemize}
Let us describe these models. The first two models deal with Ising spins with randomly distributed ferromagnetic and anti-ferromagnetic bonds $J_{ik}$. For SK model, given a configuration of $N$ Ising spins,
\begin{equation*}
\sigma =(\sigma_1,\ldots,\sigma_n)\in\{-1,+1\}^N
\end{equation*}
the Hamiltonian of the model is given by
\begin{equation*}
H_J^{SK}(\sigma)=\frac{1}{\sqrt{N}}\sum_{i,j} J_{i,j}\sigma_i \sigma_j.
\end{equation*}
where $(J_{i,j})$ are i.i.d standard Gaussian random variables.
In EA model, the sum runs only over nearest neighbour bonds $<ij>$ on the lattice:
\begin{equation*}
H_J^{EA}(\sigma)=\sum_{<ij>} J_{ij}\sigma_i \sigma_j.
\end{equation*}
In the RFIM, we consider constant  interactions $J$ with local magnetic fields $h=(h_1,h_2,\ldots,h_N)$:
\begin{equation*}
H_h= J\sum_{<ik>} \sigma_i\sigma_k+\sum_{i}^{N}h_i\sigma_i.
\end{equation*}
In all these three models, the normalized partition functions are defined by
\begin{equation*}
W_N= \frac{E\big(e^{\beta H_N}\big)}{\bbE E\big(e^{\beta H_N}\big)}
\end{equation*}
Here $\bbE$ is the measure of the random environment and $E$ is the measure of the configurations $s=(s_1,s_2,\ldots,s_N)$. 
Let us provide a proof for the SK model.
\begin{proof}[Proof for SK model]
	The proof is similar as in the directed polymer case. For a fixed configuration $\sigma$, we define a new random environment $J^\sigma$ which is given by 
	$$J_{ik}^\sigma= J_{ik}\sigma_i\sigma_k$$
	Then we can define the probability measure $Q^\sigma$
	\begin{equation*}
	\frac{dQ^{\sigma}}{d\bbP}=e^{\frac{\beta}{\sqrt{N}} \sum_{ik} J^{\sigma}_{ik}-\frac{(N-1)}{2}\lambda(\beta)}.
	\end{equation*}
	where $\lambda(\beta)= \log \bbE(e^{\beta J_{ik}}) <\infty$.
	Under $Q^\sigma$, the random variables $J^\sigma_{ik}$ are still independent, then they are positively associated and these two following applications are increasing for $\beta \geq 0$:
	\begin{align*}
	(J^\sigma_{ik})&\mapsto \sum_{ik} J^\sigma_{ik},\\
	(J^\sigma_{ik})&\mapsto \Phi'(e^{\frac{\beta}{\sqrt{N}} \sum_{ik} J^{\sigma}_{ik}-\frac{(N-1)}{2}\lambda(\beta)}).
	\end{align*} 
	These facts can easily yield the desired result as in the previous proof.
\end{proof}
We come back to the polymer case and consider the problem of finding a martingale corresponding to the normalized partition function $W_n(\beta)$ in certain cases. Let first consider the Bernoulli environment, i.e. when $\omega(i,x)$ are  symmetric Bernoulli random variables. In that case is elementary to check that:
\begin{equation}\label{re:ber}
\frac{e^{\beta \omega(i,x)}}{\bbE(e^{\beta \omega(i,x)})}=1+\omega(i,x)\tanh \beta.
\end{equation}

And if we define the stopping time $T_\beta= \inf\{t, \vert B_t\vert =\tanh\beta\}$, then 
$$\frac{e^{\beta \omega(i,x)}}{\bbE(e^{\beta \omega(i,x)})}=1+B(T_\beta),$$
with $B$ a Brownian motion. 
Hence given a family of independent Brownian motions $\{B_{i,x},\  i\in \bbN,\ x\in \bbZ^d\}$ then we  define 
\begin{equation}
M_n(\beta):= \sum_{i=1}^n \prod_{x\in\Pi_n} \Big(1+B_{i,x_i}\big(T_\beta(i,x_i)\big)\Big),
\end{equation}  
where $\Pi_n$ is the set of paths of length $n$ and $T_\beta(i,x)$ is the stopping time associated with the Brownian motion $B_{i,x}$. Since the product of  independent martingales is still a martingale respect to the product filtration then $M_n(\beta)$ is clearly a martingale.
  
Actually the key ingredient in this construction is the identity (\ref{re:ber}), which can be generalized with other choices of distributions. One can find this list of distribution in the Section Open Problems in 
\cite{HCY11} (page 359).

{\bf Acknowledgement.} 
The author would like to thank his thesis advisor Francis Comets for his support  and  comments.

\end{document}